\UseRawInputEncoding

\documentclass[10pt]{amsart}

\newtheorem{theorem}{Theorem}
\newtheorem{lemma}{Lemma}

\newtheorem{cor}{Corollary}
\newtheorem{prop}{Proposition}

\usepackage{amsmath}
\usepackage[psamsfonts]{amssymb}
\usepackage[mathscr]{euscript}

\newcommand{\cal}{\mathcal}

\title[] {Gray-Hervella classes on product twistor spaces}
\author{Johann Davidov}

\thanks{The author is partially supported by  the National Science
Fund, Ministry of Education and Science of Bulgaria under contract
KP-06-N82/6}

\address{Institute of Mathematics and Informatics \\
Bulgarian Academy of Sciences\\ Acad. G.Bonchev Str. Bl.8\\
1113 Sofia\\ Bulgaria} \email{jtd@math.bas.bg}

\begin{document}

\begin{abstract}

Motivated by generalized geometry (in the sense of Hitchin), the product bundle ${\mathcal Z}\times_{M} {\mathcal Z}$ of the twistor space ${\mathcal Z}$ of a Riemannian manifold $(M,g)$ is considered.  The product twistor space admits a natural family of Riemannian metrics and four compatible almost complex structures, analogs of the Atiyah-Hitchin-Singer  and Eells-Salamon almost complex structures on the twistor space. The Gray-Hervellal classes of these almost Hermitian structures are determined in the case when the dimension of the base manifold $M$ is four.

 \vspace{0,1cm} \noindent 2010 {\it Mathematics Subject Classification} 53C28; 53C55; 53C15.

\vspace{0,1cm} \noindent {\it Key words: twistor spaces, almost
complex structures, Gray-Hervella classes}.

\end{abstract}

\maketitle \vspace{0.5cm}

\section{Introduction}

The Gray-Hervella classification \cite{GH} of almost Hermitian structures is widely recognized as a useful tool in studying  these structures. In the present paper, the possible Gray-Hervella classes are determined on the product bundle ${\mathcal Z}\times_{M} {\mathcal Z}\to M$ (with fibre the product of fibres at the same point) of the twistor space ${\mathcal Z}$ of a Riemannian manifold $(M,g)$.  Recall that the latter space is the bundle over $M$ whose fibre at a point $p\in M$ consists of $g$-orthogonal complex structures on the tangent space $T_pM$. The motivation for considering the product twistor space  ${\mathcal Z}\times_{M} {\mathcal Z}$ comes from the generalized complex geometry in the sense of N. Hitchin \cite{Hit02, Gu}. Given a generalized Riemannian metric on a smooth manifold $M$, one can consider the bundle ${\mathcal G}$ over $M$ whose fibre at a point $p$ is the space of generalized complex structures on $T_pM$ compatible with the generalized metric. This bundle, called the generalized twistor space of $M$, is studied in \cite{D019}. The generalized metric yields a Riemannian metric $g$ on $M$ and a $g$-metric connection $D$ with totally skew-symmetric torsion. If ${\mathcal Z}$ is the usual twistor space of the Riemannian manifold $(M,g)$, then the generalized twistor space ${\mathcal G}$ is diffeomorphic to the product twistor space  ${\mathcal Z}\times_{M} {\mathcal Z}$.
In \cite{D019}, using this observation and the connection $D$, four generalized almost complex structures are defined on the generalized twistor space. One of them is an analog of the Atiyah-Hitchin-Singer \cite{AHS} almost complex structure, the others of the Eells-Salamon \cite{ES} one on the usual twistor space. These generalized almost complex structures are not induced by usual almost complex structures, although the generalized twistor space admits four  almost complex structures also defined by means of the diffeomorphism ${\mathcal G}\cong {\mathcal Z}\times_{M} {\mathcal Z}$ and the connection $D$. The integrability conditions for the four generalized almost complex structures are found in \cite{D019}, the integrability conditions for the four usual almost complex structures are established in \cite{D020} in the case when the dimension of $M$ is four, the most interesting case. These almost complex structures are compatible with a natural family of Riemannian metrics on ${\mathcal Z}\times_{M} {\mathcal Z}$. In this paper, the Gray-Hervella classes of the corresponding almost Hermitian structures are found under the natural assumption that the torsion of the connection $D$ vanishes, i.e. $D$ is the Levi-Civita connection of $g$ (this is one of the integrability conditions for the analog of the Atiyah-Hitchin-Singer structure on ${\mathcal G}$).

\bigskip

\section{Basics of twistor spaces}

\subsection{The manifold of compatible linear complex structures}

Let $T$ be a $2m$-dimensional real vector space endowed with an Euclidean metric $g$. Recall that a complex structure $J$
on $T$ is called compatible with $g$ if it is $g$-orthogonal:  $g(JX,JY)=g(X,Y)$. The set $Z=Z(T,g)$ of all compatible complex structures on $T$
has a natural smooth structure of an embedded submanifold of the vector space $\mathfrak{so}(g)$ of skew-symmetric
endomorphisms of $(T,g)$.

The group of orthogonal transformations of $(T,g)$ acts smoothly and transitively on the set $Z$ by conjugation. The
isotropy subgroup at a fixed $J\in Z$ consists of the orthogonal transformations commuting with $J$. Therefore $Z$ can be
identified with the homogeneous space $O(2m)/U(m)$. In particular, $dim\,Z=m^2-m$. Note also that  $Z$ has two connected
components. If we fix an orientation on $T$, these components consist of all complex structures on $T$ compatible with
the metric $g$ and inducing $\pm$ the orientation of $T$; each of them has the homogeneous representation $SO(2m)/U(m)$.

The tangent space at a point $J\in Z$ is $T_JZ=\{V\in\mathfrak{so(g)}:~ JV+VJ=0\}$. The smooth manifold $Z$ admits an
almost complex structure ${\mathscr K}$ defined by ${\mathscr K}V=J\circ V$. The almost complex structure ${\mathscr K}$
is compatible with the restriction $G$ to $Z$ of the standard metric $-\frac{1}{2}Trace\,AB$ of $\mathfrak{so}(g)$.

Given an orthonormal basis  $e_1,...,e_{2m}$  of $T$, denote by $S_{a,b}$, $a,b=1,...,2m$, the skew-symmetric
endomorphisms of $T$ defined by
$$
S_{a,b}e_c=\delta_{ac}e_b-\delta_{bc}e_a,\quad c=1,...,2m.
$$
Then  $\{S_{ab}:a<b\}$ is a $G$-orthonormal basis of $\mathfrak{so}(g)$.

Let $J\in Z$ and let $e_1,...,e_{2m}$ be an orthonormal basis of $T$ such that $Je_{2i-1}=e_{2i}$, $i=1,...,m$. Set
$$
\begin{array}{c}
A_{r,s}=\frac{1}{\sqrt 2}(S_{2r-1,2s-1}-S_{2r,2s}),\quad
B_{r,s}=\frac{1}{\sqrt 2}(S_{2r-1,2s}+S_{2r,2s-1}),\\[6pt]
r=1,...,m-1,\>s=r+1,...,m.
\end{array}
$$
Then $\{A_{r,s},B_{r,s}\}$ is a $G$-orthonormal basis of $T_{J}Z$ with $B_{r,s}={\mathscr K}A_{r,s}$.

Denote by $D$  the Levi-Civita connection of the metric $G$ on $Z$. Let $X,Y$ be vector fields on $Z$ considered as
$\frak{so}(g)$-valued functions on $\frak{so}(g)$. By the Koszul formula for the Levi-Civita connection, for every $J\in
Z$,
\begin{equation}\label{coder}
(D_{X}Y)_J=\frac{1}{2}(Y'(J)(X_J)+J\circ Y'(J)(X_J)\circ J)
\end{equation}
where $Y'(J)\in Hom(\frak{so}(g),\frak{so}(g))$ is the derivative of the function $Y:\frak{so}(g)\to \frak{so}(g)$ at the
point $J$. The latter formula easily implies that $D{\mathscr K}=0$, so $(G,{\mathscr K})$ is a K\"ahler structure on
$Z$. Note also that the metric $G$ is Einstein with scalar curvature $\frac{m}{2}(m-1)^2$ (see, for example, \cite{D05}).

In the case when $T$ is of dimension four, each of the two connected components $Z$ can be identified with a $2$-sphere.
It is often convenient to describe this identification in terms of the space $\Lambda^2T$.

The metric $g$ of $T$ induces a metric on $\Lambda^2T$ given by
\begin{equation}\label{g^2}
g(x_1\wedge x_2,x_3\wedge x_4)=g(x_1,x_3)g(x_2,x_4)-g(x_1,x_4)g(x_2,x_3).
\end{equation}
Consider the isomorphism $\mathfrak{so}(g)\cong \Lambda^2T$ sending $\varphi\in \mathfrak{so}(g)$ to the $2$-vector
$\varphi^{\wedge}$ for which
$$
g(\varphi^{\wedge},x\wedge y)=g(\varphi x,y),\quad x,y\in {\mathbb R}^4.
$$
This isomorphism is an isometry with respect to the metric $G$ on $\mathfrak{so}(g)$ and the metric $g$ on
$\Lambda^2{\mathbb R}^4$.

Fix an orientation on $T$ and denote by  $Z_{\pm}$ the set of complex structures on $T$ compatible with the metric $g$
and inducing $\pm$ the orientation of $T$.

If $dim\,T=4$, the Hodge star operator defines an endomorphism $\ast$ of $\Lambda^2T$ with $\ast^2=Id$. Hence we have the
orthogonal decomposition
$$
\Lambda^2T=\Lambda^2_{-}T\oplus\Lambda^2_{+}T,
$$
where $\Lambda^2_{\pm}T$ are the subspaces of $\Lambda^2T$ corresponding to the $(\pm 1)$-eigenvalues of the operator
$\ast$. Let $(e_1,e_2,e_3,e_4)$ be an oriented orthonormal basis of ${\mathbb R}^4$. Set
\begin{equation}\label{s-basis}
s_1^{\pm}=\frac{1}{\sqrt 2}(e_1\wedge e_2 \pm e_3\wedge e_4), \quad s_2^{\pm}=\frac{1}{\sqrt 2}(e_1\wedge e_3\pm
e_4\wedge e_2), \quad s_3^{\pm}=\frac{1}{\sqrt 2}(e_1\wedge e_4\pm e_2\wedge e_3).
\end{equation}
Then $(s_1^{\pm},s_2^{\pm},s_3^{\pm})$ is an orthonormal basis of $\Lambda^2_{\pm}{\mathbb R}^4$.

The isomorphism $\varphi\to \varphi^{\wedge}$ identifies $Z_{\pm}$ with the sphere of radius $\sqrt 2$ in the Euclidean
vector space $(\Lambda^2_{\pm}T,g)$. Under this isomorphism, if $J\in Z_{\pm}$, the tangent space $T_{J}Z=T_{J}Z_{\pm}$
is identified with the orthogonal complement $({\mathbb R} J)^{\perp}$ of the space ${\mathbb R}J$ in
$\Lambda^2_{\pm}{\mathbb R}^4$.

The orientation of $\Lambda^2_{\pm}T$ defined by the basis (\ref{s-basis}) does not depend on the choice of the basis
$(e_1,...,e_4)$ (\cite{D017}). Consider the $3$-dimensional Euclidean space $(\Lambda^2_{\pm}T,g)$ with
this orientation and denote by $\times$ the usual vector-cross product on it. For $\sigma,\tau\in\Lambda^2_{\pm}T$,
denote by $K_{\sigma}$ and $K_{\tau}$ the endomorphisms of $T$ corresponding to $\sigma$ and $\tau$ under the isomorphism
$\Lambda^2T\cong \mathfrak{so}(g)$. Then $\sigma\times\tau$ corresponds to $\pm\frac{1}{\sqrt 2}[K_{\sigma},K_{\tau}]$.
Thus, if $J\in Z_{\pm}$ and $V\in T_{J}Z=T_{J}Z_{\pm}$,
$$
({\mathcal K}V)^{\wedge}=\pm\frac{1} {\sqrt 2}( J^{\wedge}\times  V^{\wedge}).
$$

\subsection {The twistor space of a Riemannian manifold}

Let $(M,g)$ be a Riemannian manifold, $dim\,M=2m$.  Denote by ${\mathcal Z}={\mathcal Z}(TM,g)$ the bundle over $M$ whose
fibre at a point $p\in M$ is $Z(T_pM,g_p)$, the space of complex structures on $T_pM$ compatible with the metric $g_p$
(the usual twistor space of $(M,g)$). Consider ${\mathcal Z}$ as a subbundle of the bundle ${\mathcal A}={\mathcal
A}(TM,g)$ of $g$ - skew-symmetric endomorphisms of $TM$. Denote  the projection of the bundle ${\mathcal A}$ onto $M$ and
its restriction to ${\mathcal Z}$ by $\pi$.  Then the vertical space ${\mathcal V}_J$ of ${\mathcal Z}$ at a point $J\in
{\mathcal Z}$ consists of skew-symmetric endomorphisms $V$ of $(T_{\pi(J)},g)$ anti-commuting with $J$. Let $\nabla$ be
the Levi-Civita connection of $(M,g)$. The induced connection on the bundle ${\mathcal A}$ will also be denoted by
$\nabla$. Let ${\mathcal H}_J$ be the horizontal subspace of $T_J{\mathcal A}$ with respect to $\nabla$. Take an
orthonormal basis $e_1,...,e_{2m}$ of $T_pM$, $p=\pi(J)$, such that $Je_{2k-1}=e_{2k}$, $Je_{2k}=-e_{2k-1}$, $k=1,...,m$.
Extend this basis to a frame of vector fields $E_1,...,E_{2m}$ such that $\nabla E_i|_p=0$ and define a section $K$ of
${\mathcal A}$ by $KE_{2k-1}=E_{2k}$, $KE_{2k}=-E_{2k-1}$. Obviously $K$ takes values in ${\mathcal Z}$ and $\nabla
K|_p=0$. Hence the horizontal space ${\mathcal H}_J=K_{\ast}(T_pM)$ is tangent to ${\mathcal Z}$. Thus we have the
decomposition $T{\mathcal Z}={\mathcal V}\oplus {\mathcal H}$. Using this decomposition, we can define a $1$-parameter
family  $h_t$, $t>0$, of Riemannian metrics on ${\mathcal Z}$ by
$$
h_t(X^h+V,Y^h+W)=g(X,Y)+tG(V,W)
$$
where $X^h,Y^h$ are the horizontal lifts of $X,Y\in TM$, $V,W$ are vertical  vectors, and $G$ is (the restriction) of the
standard metric $-\frac{1}{2}Trace_g\,AB$ of ${\mathcal A}$. By the Vilms theorem \cite{V} (or \cite[Theorem
9.59]{Besse})  the projection map $\pi:({\cal Z},h_t)\to (M,g)$ is a Riemannian submersion with totally geodesic fibres
(this can also be proved by a direct computation). Recall also that the metrics $h_t$ are compatible with the
Atiyah-Hitchin-Singer \cite {AHS} and Eeels-Salamon \cite{ES} almost complex structures ${\mathcal J}_1$ and ${\mathcal
J}_2$ defined by
$$
{\mathcal J}_1X^h_J={\mathcal J}_2X^h_J=(JX)^h_J, \quad {\mathcal J}_1V=-{\mathcal J}_2V=J\circ V,\quad X\in
T_{\pi(J)}M,\quad V\in{\mathcal V}_J.
$$

\subsection {The product twistor space}

Now, let ${\mathcal G}$ be the product bundle ${\mathcal Z}\times_{M}{\mathcal Z}$. We shall consider ${\mathcal G}$ as a
subbundle of the vector bundle $\pi:{\mathcal A}\oplus {\mathcal A}\to M$ endowed with the connection induced by the
Levi-Civita connection of $(M,g)$, both denoted by $\nabla$.  The vertical space of ${\mathcal G}$ at a point
$J=(J_1,J_2)$ consists of the pairs $(V_1,V_2)$ of skew-symmetric endomorphisms $V_1,V_2$ of $(T_{\pi(J)}M,g)$
anti-commuting with $J_1$, $J_2$, respectively. The horizontal spaces of ${\mathcal A}\oplus {\mathcal A}$ are tangent to
${\mathcal G}$ and we have the decomposition $T{\mathcal G}={\mathcal H}\oplus {\mathcal V}$.

 For $t_1>0$, $t_2>0$, set ${\bf t}=(t_1,t_2)$ and define a Riemannian metric on
the manifold  ${\mathcal G}$ by
$$
H_{\bf t}(X^h+V,Y^h+W)=g(X,Y)+t_1G(V_1,W_1)+t_2G(V_2,W_2)
$$
where $X,Y\in TM$ and $V=(V_1,V_2)$, $W=(W_1,W_2)$ are vertical vectors. Then the projection map $\pi:({\mathcal
G},H_{\bf t})\to (M,g)$ is a Riemannian submersion.

Any fibre ${\mathcal V}_J$ of ${\mathcal G}$ possesses four complex structures ${\mathscr K}_n$ for which $(G_{\bf
t}|{\mathcal V}_J,{\mathscr K}_n)$ are K\"ahler structures on the fibre. These are defined by
$$
\begin{array}{c}
{\mathscr K}_1(V_1,V_2)=(J_1\circ V_1,J_2\circ V_2),\quad {\mathscr K}_2(V_1,V_2)=(J_1\circ V_1,-J_2\circ V_2)\\[8pt]
{\mathscr K}_3=-{\mathscr K}_2,\quad {\mathscr K}_4=-{\mathscr K}_1.
\end{array}
$$

Correspondingly, we can define four almost complex structures ${\mathscr J}^n$ on ${\mathcal G}$ setting for
$J=(J_1,J_2)\in{\mathcal G}$, $X\in T_{\pi(J)}M$, and $V\in {\mathcal V}_J$
$$
{\mathscr J}^nX^h_J=(J_1X)^h_J,\quad {\mathscr J}^nV={\mathscr K_n}V,\quad n=1,...,4.
$$
Clearly, these almost complex structures are compatible with the metrics $H_{\bf t}$.

\section{Technical lemmas}

Let $(\mathscr{U},x_1,...,x_{2m})$ be a local coordinate system of $M$ and $\{E_1',...,E_{2m}'\}$,
$\{E_1'',...,E_{2m}''\}$ orthonormal frames of $TM$ on $\mathscr{U}$, respectively. Define sections $S_{ij}'$,
$S_{ij}''$, $1\leq i,j\leq {2m}$, of the bundle ${\mathcal A}$  by the formulas
\begin{equation}\label{eq Sij}
S_{ij}'E_k'=\delta_{ik}E_j' - \delta_{kj}E_i',\quad S_{ij}''E_k''=\delta_{ik}E_j'' - \delta_{kj}E_i'',
\end{equation}

For $a=(a',a'')\in {\mathcal A}\oplus {\mathcal A}$, set
\begin{equation}\label{coord}
\tilde x_{i}(a)=x_{i}\circ\pi(a),\quad y_{kl}'(a)=G(a',S_{kl}'\circ\pi(a)),\quad y_{kl}''(a)=G(a'',S_{kl}''\circ\pi(a)).
\end{equation}
Then $((\tilde x_{i}),(y_{jk}'), (y_{rs}''))$, $1\leq i\leq 2m$, $1\leq j < k\leq 2m$, $1\leq r < s\leq 2m$, is a local coordinate system on the
total space of the bundle ${\mathcal A}\oplus {\mathcal A}$.

\smallskip

If
\begin{equation}\label{V}
V=\sum_{1\leq j<k\leq 2m}[v_{jk}'\frac{\partial}{\partial y_{jk}'}(J)+v_{jk}''\frac{\partial}{\partial y_{jk}''}(J)].
\end{equation}
is a vertical vector of ${\mathcal G}$ at a point $J$, the vertical vector ${\mathscr J}^nV$ is given by
\begin{equation}\label{cal J/ver}
{\mathscr J}^{n}V=(-1)^{n+1}\sum_{j<k} \sum_{s=1}^{2m}[\pm v_{js}'y_{sk}'\frac{\partial}{\partial y_{jk}'}+
v_{js}''y_{sk}''\frac{\partial}{\partial y_{jk}''}]
\end{equation}
where the plus sign corresponds to $n=1,4$ and the minus sign to $n=2,3$.

Note also that, for every $X\in TM$,  we have
\begin{equation}\label{cal J/hor}
\begin{array}{c}
 {\mathscr J}^nX^h=\sum\limits_{i,j=1}^{2m}(g(X,E_i')\circ\pi)y_{ij}'E_j^{'h}.
\end{array}
\end{equation}

 It is convenient to set $v_{ij}'=-v_{ji}'$, $v_{ij}''=-v_{ji}''$  and
$y_{ij}'=-y_{ji}'$, $y_{ij}''=-y_{ji}''$ for $i\geq j$, $1\leq i,j\leq {2m}$. Then the endomorphism $V$ of $T_{p}M\oplus
T_{p}M$, $p=\pi(J)$, is determined by
$$
VE_i'=\sum_{j=1}^{2m} v_{ij}'E_j',\quad VE_i''=\sum_{j=1}^{2m} v_{ij}''E_j''.
$$

For a vector field
$$X=\sum_{i=1}^{2m} X^{i}\frac{\partial}{\partial x_i}$$
on $\mathscr{U}$, the horizontal lift $X^h$ on $\pi^{-1}(\mathscr{U})$ is given by
\begin{equation}\label{Xh}
\begin{array}{c}
X^{h}=\displaystyle{\sum_{k=1}^{2m} (X^{l}\circ\pi)\frac{\partial}{\partial\tilde
x_l}}\\[8pt]
- \displaystyle{\sum_{i<j}\sum_{k<l} [y_{kl}'(G(\nabla_{X}S_{kl}',S_{ij}')\circ\pi)\frac{\partial}{\partial y_{ij}'} +
y_{kl}''(G(\nabla_{X}S_{kl}'',S_{ij}'')\circ\pi)\frac{\partial}{\partial y_{ij}''}}]
\end{array}
\end{equation}

   Let $a=(a',a'')\in {\mathcal A}\oplus {\mathcal A}$ and $p=\pi(a)$. Denote by $A(T_pM)$ the space of skew-symmetric endomorphisms of
$(T_pM,g_p)$, the fibre of the bundle ${\mathcal A}$ at the point $p$.  Then (\ref{Xh}) implies that, under the standard
identification of $T_{a}(A(T_pM)\oplus A(T_{p}M))$ with the vector space $A(T_pM)\oplus A(T_pM)$, we have
\begin{equation}\label{[XhYh]}
[X^{h},Y^{h}]_{a}=[X,Y]^h_a + R(X,Y)a
\end{equation}
where $R(X,Y)a=R(X,Y)a'+R(X,Y)a''$ is the curvature of the connection $\nabla$ on the vector bundle ${\mathcal A}\oplus {\mathcal A}$
(for the curvature tensor we adopt the
following definition: $R(X,Y)=\nabla_{[X,Y]}- [\nabla_{X},\nabla_{Y}]$). Note also that (\ref{V}) and (\ref{Xh}) imply
the well-known fact that
\begin{equation}\label{[V,Xh]}
[V,X^h]~\rm{is~ a~ vertical~ vector~ field}.
\end{equation}

\smallskip

\noindent {\it Notation}. Let $J=(J_1,J_2)\in {\mathcal G}$ and $p=\pi(J)$. Take orthonormal bases $\{e_1',...,e_{2m}'\}$,
$\{e_1'',...,e_{2m}''\}$ of $T_pM$ such that $e_{2l}'=J_1e_{2l-1}'$, $e_{2l}''=J_2e_{2l-1}''$ for $l=1,...,m$. Let
$\{E_i'\}$, $\{E_i''\}$, $i=1,...,2m$, be  orthonormal frames of $TM$ near the point $p$ such that
$$
E_i'(p)=e_i',\> E_i''(p)=e_i'' ~\mbox { and }~ \nabla\, E_i'|_p=0,\> \nabla\,E_i''|_p=0, \quad i=1,...,2m.
$$

 For any (local) section $a=(a',a'')$ of ${\mathcal A}\oplus {\mathcal A}$, 
denote by $\widetilde a$ the vertical vector field on ${\cal G}$ defined by
\begin{equation}\label{eq tilde a}
\widetilde a_J=(a'_{\pi(J)}+J_1\circ a'_{\pi(J)}\circ J_1, a''_{\pi(J)}+J_2\circ a''_{\pi(J)}\circ J_2),\quad
J=(J_1,J_2).
\end{equation}
$\hfill\Box$

\smallskip

If $S'_{ij}$ and $S_{ij}''$ are the sections of bundle ${\mathcal A}$  defined by (\ref{eq Sij}), then clearly $\nabla S_{ij}'|_p=\nabla S_{ij}|_p=0$, $1\leq i,j\leq {2m}$. Note also that for every $J\in {\mathcal G}$ we can find sections $a_1,...,a_l$, $l=2(m^2-m)$, of ${\mathcal A}\oplus
{\mathcal A}$ near the point $p=\pi(J)$ such that $\widetilde a_1,...,\widetilde a_l$ form a basis of the vertical vector
space at each point in a neighbourhood of $J$.

\vspace{0.1cm}

The next lemma is a kind of folklore appearing in different contexts (cf, for example,\cite{G,D05,D017}),

\begin{lemma}\label{Lie-hor-ver}
Let $X$ be a vector field on $M$ and let $a=(a',a'')$ be a section of ${\mathcal A}\oplus{\mathcal A}$  defined on a
neighbourhood of the point $p=\pi(J)$. Then:
$$
[X^h,\widetilde{a}]_{J}=\widetilde{(\nabla_X a)}_{J},\quad [X^h,{\mathscr J}^n\widetilde{a}]_{J}={\mathscr
K}_n\widetilde{(\nabla_X a)}_{J}.
$$
\end{lemma}

\begin{proof}
Let $a'(E_i')=\sum\limits_{j=1}^{2m}a_{ij}'E_j'$, $a''(E_i'')=\sum\limits_{j=1}^{2m}a_{ij}''E_j''$,\, $i=1,...,2m$. Then, in the local
coordinates of ${\mathcal A}\oplus {\mathcal A}$ introduced above,
$$
\widetilde a=\sum_{i<j}[\widetilde a_{ij}'\frac{\partial}{\partial y_{ij}'}+\widetilde a_{ij}''\frac{\partial}{\partial
y_{ij}''}]
$$
where
\begin{equation}\label{tilde aij}
\widetilde a_{ij}'=a_{ij}'\circ\pi+\sum_{k,l=1}^{2m}y_{ik}'(a_{kl}'\circ\pi)y_{lj}',\quad \widetilde
a_{ij}''=a_{ij}''\circ\pi+\sum_{k,l=1}^{2m}y_{ik}''(a_{kl}''\circ\pi)y_{lj}''.
\end{equation}
By (\ref{Xh}), we have for every vector field $X$ on $M$ near the point $p$
\begin{equation}\label{X-hor-yij}
\begin{array}{c}
X_J^h=\displaystyle{\sum_{i=1}^{2m} X^i(p)\frac{\partial}{\partial\tilde
x_i}(J)},\\[8pt]
\displaystyle{[X^h,\frac{\partial}{\partial y_{ij}'}]_J=[X^h,\frac{\partial}{\partial y_{ij}''}]_J}=0
\end{array}
\end{equation}
since $\nabla S_{ij}'|_p=\nabla S_{ij}''|_p=0$. Moreover,
\begin{equation}\label{Xa_ij}
(\nabla_{X_p}a')(E_i')=\sum_{j=1}^{2m}X_p(a_{ij}')E_j', \quad (\nabla_{X_p}a'')(E_i'')=\sum_{j=1}^{2m}X_p(a_{ij}'')E_j''
\end{equation}
since $\nabla E_i'|_p=\nabla E_i''|_p=0$. Now, the lemma follows by a simple computation.
\end{proof}

\smallskip

\noindent {\it Notation}.  Further on, we shall often make use of the standard isometric isomorphism ${\mathcal
A}\cong\Lambda^{2}TM$ that assigns to each $a\in A(T_{p}M)$ the 2-vector $a^{\wedge}$ for which
$$
g(a^{\wedge},X\wedge Y)=g(aX,Y),\quad X,Y\in T_{p}M,
$$
where the metric on $\Lambda^{2}TM$ is given by
$$
g(X_1\wedge X_2,X_3\wedge X_4)= g(X_1,X_3)g(X_2,X_4)-g(X_1,X_4)g(X_2,X_3).
$$
Under this isomorphism the connection on ${\mathcal A}$ induced by the Levi-Civita connection $\nabla$ of $(M,g)$ goes to
the connection on $\Lambda^{2}TM$ induced by $\nabla$. We shall use the same symbol $\nabla$ for  these connections.

Also,  juxtaposition of endomorphisms will mean "composition".
$\hfill\Box$

\begin{lemma}\label{R[a,b]} {\rm (\cite{D05})} For every $a,b\in A(T_pM)$ and $X,Y\in T_pM$, we
have
$$
G(R(X,Y)a,b)=g(R([a,b]^{\wedge})X,Y).
$$
\end{lemma}

\begin{cor}\label{R-J}\label{R(X,Y)J,V}
If $J=(J_1,J_2)\in{\mathcal G}$, $X,Y\in\pi(J)$, and $V=(V_1,V_2)\in {\mathcal V}_J$
$$
H_{\bf t}(R(X,Y)J,V)=2g(R(t_1(J_1V_1)^{\wedge}+t_2(J_2V_2)^{\wedge})X,Y).
$$
\end{cor}

\begin{proof}
We have
$$
\begin{array}{c}
H_{\bf
t}(R(X,Y)J,V)=t_1g(R([J_1,V_1]^{\wedge})X,Y)+t_2g(R([J_2,V_2]^{\wedge})X,Y)\\[6pt]
=2g(R(t_1(J_1V_1)^{\wedge}+t_2(J_2V_2)^{\wedge})X,Y).
\end{array}
$$
\end{proof}

\noindent {\it Notation}. The curvature operator of the connection $\nabla$ will be denoted by ${\mathcal R}$. This is
the self-adjoint endomorphism of $\Lambda ^{2}TM$ defined by
$$
 g({\cal R}(X\land Y),Z\land T)=g(R(X,Y)Z,T),\quad X,Y,Z,T\in TM.
$$
Also, we denote the Levi-Civita connection of the metric $H_{\bf t}$ on ${\mathcal G}$ by $D$.
$\hfill\Box$

\begin{lemma}\label{LC}{\rm (\cite{DM91,D05})} If $X,Y$ are vector
fields on $M$ and $V$ is a vertical vector field on $Z(E)$, then:

\smallskip

\noindent $(i)$ At any point $J\in {\mathcal G}$
\begin{equation}\label{D-hh}
(D_{X^h}Y^h)_{J}=(\nabla_{X}Y)^h_{J}+\frac{1}{2}R(X,Y)J.
\end{equation}

\smallskip

\noindent $(ii)$ The fibres of ${\mathcal G}\to M$ are totally geodesic submanifolds.

\smallskip

\noindent $(iii)$ The vector $(D_{V}X^h)_{J}$ is horizontal and
\begin{equation}\label{D-vh}
H_{\bf t}(D_{V}X^h,Y^h)_J=H_{\bf t}(D_{X^h}V,Y^h)_J=-g({\mathcal R}(t_1(J_1V_1)^{\wedge}+t_2(J_2V_2)^{\wedge}),X\wedge
Y).
\end{equation}

\end{lemma}

\begin{proof}
$(i)$ Formula (\ref{D-hh}) follows from the Koszul formula and the fact that the Lie bracket of a vertical and a
horizontal vector field is a vertical vector field

\smallskip

\noindent $(ii)$ Set $S_{ij}=(S_{ij}',S_{ij}'')$ and, as in the preceding section, define
$$
\begin{array}{c}
A_{r,s}=\frac{1}{\sqrt 2}(S_{2r-1,2s-1}-S_{2r,2s}),\quad
B_{r,s}=\frac{1}{\sqrt 2}(S_{2r-1,2s}+S_{2r,2s-1}),\\[6pt]
r=1,...,m-1,\>s=r+1,...,m.
\end{array}
$$
The vertical vector fields $\widetilde{A}_{rs}$, $\widetilde{B}_{rs}$ corresponding to the sections $A_{r,s}$, $B_{r,s}$
of ${\mathcal A}\oplus{\mathcal A}$ constitute a frame of the vertical bundle ${\mathcal V}$ of ${\mathcal G}$ in a
neighbourhood of $J$. Let $V\in{\mathcal V}_J$. Take a section $Q$ of ${\mathcal A}\oplus{\mathcal A}$ such that $Q(p)=V$
and $\nabla Q|_p=0$ where $p=\pi(J)$. By Lemma~\ref{Lie-hor-ver}, $[X^h,\widetilde{A}_{rs}]_J=[X^h,\widetilde{B}_{rs}]_J=
[X^h,\widetilde{Q}]_J=0$ for every vector field $X$ in a neighbourhood of $p$. It follows from the Koszul formula for the
Levi-Civita connection that $(D_{\widetilde{Q}}\widetilde{A}_{rs})_J$ and $(D_{\widetilde{Q}}\widetilde{B}_{rs})_J$ are
$H_{\bf t}$-orthogonal to every horizontal vector $X^h_J$. Hence $D_{V}\widetilde{A}_{rs}$ and $D_{V}\widetilde{B}_{rs}$
are vertical tangent vectors of ${\mathcal G}$ at $J$. It follows that, for every vertical vector field $W$, $D_{V}W$ is
a vertical vector field. Thus the fibres of ${\mathcal G}$ are totally geodesic submanifolds. This, of course,  follows
also from the Vilms theorem \cite{V} (or \cite{Besse}).

\smallskip

\noindent $(iii)$  By $(ii)$,  $D_{V}X^h$ is orthogonal to every vertical vector field, thus $D_{V}X^h$ is a horizontal
vector field. Hence $D_{V}X^h={\cal H}D_{X^h}V$ since $[V,X^h]$ is vertical. Therefore by (\ref{[XhYh]}) and
Corollary~\ref{R(X,Y)J,V}
$$
\begin{array}{c}
H_{\bf t}(D_{V}X^h,Y^h)_{J}=H_{\bf t}(D_{X^h}V,Y^h)_{J}=-H_{\bf t}(V,D_{X^h}Y^h)_{J}
=-\displaystyle{\frac{1}{2}} H_{\bf t}(R(X,Y)J,V)\\[6pt]
=-g(R(t_1(J_1V_1)^{\wedge}+t_2(J_2V_2)^{\wedge})X,Y).
\end{array}
$$

\end{proof}

\noindent {\it Notation}. Let $\Omega_{{\bf t},n}(A,B)=H_{\bf t}({\mathscr J}^{n}A,B)$ be the fundamental $2$-form of the
almost Hermitian structure $(H_{\bf t},{\mathscr J}^{n})$ on ${\mathcal G}$, $n=1,...,4.$
$\hfill\Box$

\smallskip

\begin{prop}\label{DF}
Let $J=(J_1,J_2)\in{\mathcal G}$, $X,Y,Z\in T_{\pi(J)}M$,  and $U,V=(V_1,V_2),W\in{\mathcal V}_J$. Then
\begin{equation}\label{hhh}
(D_{Z^h}\Omega_{{\bf t},n})(X^h,Y^h)_J=0
\end{equation}
\end{prop}
\begin{equation}\label{vhh}
(D_{V}\Omega_{{\bf t},n})(X^h,Y^h)_J=g(V_1X,Y) -g({\mathcal R}(t_1(J_1V_1)^{\wedge}+t_2(J_2V_2)^{\wedge}),X\wedge
J_1Y+J_1X\wedge Y).
\end{equation}
where $(V_1,V_2)=V$.
\begin{equation}\label{hhv}
\begin{array}{c}
(D_{Z^h}\Omega_{{\bf t},n})(X^h,V)_J=(-1)^{n}g({\mathcal R}(\pm t_1V_1^{\wedge}+t_2V_2^{\wedge}), Z\wedge X)\\[6pt]
\hspace{6cm} + g({\mathcal R}(t_1(J_1V_1)^{\wedge}+t_2(J_2V_2)^{\wedge}),Z\wedge J_1X),
\end{array}
\end{equation}
where the plus sign corresponds to $n=1,4$ and the minus sign to $n=2,3$.
\begin{equation}\label{rest}
(D_{W}\Omega_{{\bf t},n})(X^h,V)=0,\quad (D_{Z^h}\Omega_{{\bf t},n})(U,V)=0,\quad (D_{W}\Omega_{{\bf t},n})(U,V)=0.
\end{equation}

\begin{proof}
Extend $X,Y$ to vector fields in a neighbourhood of $p=\pi(J)$ such that $\nabla X|_p=\nabla Y|_p=0$. Note that
$Z^h_J(y_{kl}')=Z^h_J(y_{kl}'')=0$,  $k,l=1,...,2m$,  by (\ref{Xh}).

Then, by  (\ref{cal J/hor}) and (\ref{D-hh}),
$$
\begin{array}{l}
(D_{Z^h}\Omega_{{\bf t},n})(X^h,Y^h)_J=Z^h_JH_{\bf t}({\mathscr J}^{n}X^h,Y^h)\\[8pt]
\hspace{4.5cm}+H_{\bf t}(D_{Z^h}X^h,(J_1Y)^h)_J-H_{\bf t}(D_{Z^h}Y^h,(J_1X)^h)_J\\[8pt]
=\sum\limits_{k=1}^{2m}\sum\limits_{l=1}^{2m}\big[y_{kl}'Z\big((g(X,E_k')g(Y,E_l')\big)=0.
\end{array}
$$
Also, by (\ref{cal J/hor}) and (\ref{D-vh}),
$$
\begin{array}{c}
(D_{V}\Omega_{{\bf
t},n})(X^h,Y^h)_J=\sum\limits_{k=1}^{2m}\sum\limits_{l=1}^{2m}V\big(y_{kl}'(g(X,E_k')g(Y,E_l'))\circ\pi\big)\\[10pt]
\hspace{4.5cm}+H_{\bf t}(D_{V}X^h,(J_1Y)^h)_J-H_{\bf t}(D_{V}Y^h,(J_1X)^h)_J\\[8pt]
=g(V_1X,Y) -g({\mathcal R}(t_1(J_1V_1)^{\wedge}+t_2(J_2V_2)^{\wedge}),X\wedge J_1Y+J_1X\wedge Y).
\end{array}
$$
Next
$$
\begin{array}{c}
(D_{Z^h}\Omega_{\bf t,n})(X^h,V)_J=H_{\bf t}(D_{Z^h}X^h,{\mathscr J}^nV)-H_{\bf t}((J_1X)^h,D_{Z^h}V)\\[6pt]
=\frac{1}{2}H_{\bf t}(R(Z,X)J,{\mathscr K}_nV)+g({\mathcal R}(t_1(J_1V_1)^{\wedge}+t_2(J_2V_2)^{\wedge}),Z\wedge J_1X)
\end{array}
$$
and (\ref{hhv}) follows from Corollary~\ref{R(X,Y)J,V}

 Next, by Lemma~\ref{LC}
$$
(D_{W}\Omega_{{\bf t},n})(X^h,V)=0.
$$
In order to prove the second identity in (\ref{rest}), take sections $a=(a',a'')$ and $b=(b',b'')$ of ${\mathcal
A}\oplus{\mathcal A}$ such that $a(p)=U$, $b(p)=V$ and $\nabla a|_p=\nabla|_p=0$. Denote by $\widetilde a$ and
$\widetilde b$ the vertical vector fields on ${\mathcal G}$ determined by $a$ and $b$ via (\ref{eq tilde a}). Then, by
Lemma~\ref{LC} $(iii)$ and Lemma~\ref{Lie-hor-ver}
$$
\begin{array}{c}
(D_{Z^h}\Omega_{{\bf t},n})(U,V)=\frac{1}{4}[Z^h_JH_{\bf t}({\mathscr J}^n\widetilde a,\widetilde b)+H_{\bf t}([Z^h,\widetilde
a],{\mathscr J}^n\widetilde b)_J-H_{\bf t}({\mathscr J}^n\widetilde a,[Z^h,\widetilde b])_J] \\[6pt]
=\frac{1}{4}Z^h_JH_{\bf t}({\mathscr J}^n\widetilde a,\widetilde b)
\end{array}
$$
Set $a'(E_i')=\sum\limits_{j=1}^{2m}a_{ij}'E_j'$, $a''(E_i'')=\sum\limits_{j=1}^{2m}a_{ij}''E_j''$,\, $i=1,...,2m$, and define
$b_{ij}'$, $b_{ij}''$ in a similar way. We have
$$
\widetilde a=\sum_{i<j}[\widetilde a_{ij}'\frac{\partial}{\partial y_{ij}'}+\widetilde a_{ij}''\frac{\partial}{\partial
y_{ij}''}],\quad \widetilde b=\sum_{i<j}[\widetilde b_{ij}'\frac{\partial}{\partial y_{ij}'}+\widetilde
b_{ij}''\frac{\partial}{\partial y_{ij}''}],
$$
where $\widetilde a_{ij}', \widetilde a_{ij}''$, $\widetilde b_{ij}', \widetilde b_{ij}''$ are given by (\ref{tilde
aij}). Then $Z^h_J(\widetilde a_{ij}')=Z^h_J(\widetilde a_{ij}'')=0$, $Z^h_J(\widetilde b_{ij}')=Z^h_J(\widetilde
b_{ij}'')=0$. Hence, by (\ref{cal J/ver}),
$$
Z^h_JH_{\bf t}({\mathscr J}^n\widetilde a,\widetilde b)=(-1)^{n+1}\sum_{j<k} \sum_{s=1}^{2m}Z^h_j[\pm \widetilde
a_{js}'y_{sk}'\widetilde b_{jk}'+ \widetilde a_{js}''y_{sk}''\widetilde b_{jk}'']=0.
$$
This proves the second identity in (\ref{rest}). The third one follows from the fact that the fibres of ${\mathcal G}$
are totally geodesic and K\"ahler.
\end{proof}

\begin{cor}\label{diff F}
Let $J=(J_1,J_2)\in{\mathcal G}$, $X,Y,Z\in T_{\pi(J)}M$,  and $U,V=(V_1,V_2),W\in{\mathcal V}_J$. Then
$$
d\Omega_{{\bf t},n}(X^h,Y^h,Z^h)=0.
$$
$$
d\Omega_{{\bf t},n}(X^h,Y^h,V)_J=g(V_1X,Y)+2(-1)^{n}g({\mathcal R}(\pm t_1V_1^{\wedge}+t_2V_2^{\wedge}), X\wedge Y),
$$
where the plus sign corresponds to $n=1,4$ and the minus sign to $n=2,3$.
$$
d\Omega_{{\bf t},n}(X^h,V,W)=0,\quad d\Omega_{{\bf t},n}(U,V,W)=0
$$
\end{cor}
\begin{cor}\label{delta F}
Let $J=(J_1,J_2)\in{\mathcal G}$, $X\in T_{\pi(J)}M$,  and $V=(V_1,V_2)\in{\mathcal V}_J$. Then
$$
(\delta \Omega_{{\bf t},n})(X^h_J)=0,\quad (\delta \Omega_{{\bf t},n})(V)=-2g({\mathcal
R}(t_1(J_1V_1)^{\wedge}+t_2(J_2V_2)^{\wedge}),J_1^{\wedge}),
$$
\end{cor}

Denote by ${\mathscr N}_n$ the Nijenhuis tensor of the almost complex structure ${\mathscr J}^n$:
$$
{\mathscr N}_n(A,B)=-[A,B]+[{\mathscr J}^nA,{\mathscr J}^nB]-{\mathscr J}^n[{\mathscr J}^nA,B]-{\mathscr J}^n[A,
{\mathscr J}^nB],\quad A,B\in T{\mathcal G}.
$$
The Nejensuis tensor is related to the covariant derivatives of the $2$-form $\Omega_{{\bf t},n}$ by the following well-known
identity
$$
\begin{array}{c}
H_{\bf t}({\mathscr N}_n(A,B),C)=(D_{A}\Omega_{{\bf t},n})({\mathscr J}^nB,C)-(D_{B}\Omega_{{\bf t},n})({\mathscr J}^nA,C)\\[6pt]
\hspace{7cm} +(D_{{\mathscr J}^nA})(B,C)-(D_{{\mathscr J}^nB}\Omega_{{\bf t},n})(A,C)
\end{array}
$$
This and Proposition~\ref{DF} imply:
\begin{cor}\label{Nij}
Let $J=(J_1,J_2)\in{\mathcal G}$, $X,Y,Z\in T_{\pi(J)}M$,  and $V=(V_1,V_2),W\in{\mathcal V}_J$. Then
$$
H_{{\bf t},n}({\mathscr N}_n(X^h,Y^h),Z^h)=0.
$$
$$
\begin{array}{c}
H_{{\bf t},n}({\mathscr N}_n(X^h,Y^h),V)_J=
2(-1)^n(g({\mathcal R}(\pm t_1V_1^{\wedge}+t_2V_2^{\wedge}), X\wedge J_1Y+J_1X\wedge Y)\\[6pt]
\hspace{3.5cm} -2g({\mathcal R}(t_1(J_1V_1)^{\wedge}+t_2(J_2V_2)^{\wedge}),X\wedge Y-J_1X\wedge J_1Y)
\end{array}
$$
where the plus sign corresponds to $n=1,4$ and the minus sign to $n=2,3$.
$$
\begin{array}{c}
H_{{\bf t},n}({\mathscr N}_n(X^h,V),Y^h)=0\quad {\rm for}~n=1,2,\\[6pt]
H_{{\bf t},n}({\mathscr N}_n(X^h,V),Y^h)_J=2g(J_1V_1X,Y)\quad {\rm for}~n=3,4.
\end{array}
$$
$$
H_{{\bf t},n}({\mathscr N}_n(X^h,V),W)=0,\quad {\mathscr N}_n(V.W)=0.
$$
\end{cor}

\section {The Gray-Hervella classes}

Let $\widetilde \Omega_{t,k}$, $t>0$, $k=1,2$, be the fundamental $2$-form of the almost Hermitian structure $(h_t,{\mathcal
J}_k)$ on the twistor space ${\mathcal Z}$. Let $\widetilde D$ be the Levi-Civita connection of $h_t$. Considering ${\mathcal Z}$ as a subbundle of ${\mathcal A}$, computations
similar to those in the proof of Proposition~\ref{DF} give formulas for the covariant derivatives $\widetilde D\widetilde
\Omega_{t,k}$ which show the following relation with the covariant derivative $D\Omega_{{\bf t},n}$.

Let $J=(J_1,J_2)\in{\mathcal G}$. Recall that, considering ${\mathcal G}$ as a subbundle of ${\mathcal A}\oplus {\mathcal A}$, the horizontal space of ${\mathcal G}$ at $J$ is the horizontal space of the Whitney sum ${\mathcal A}\oplus {\mathcal A}$ at $J$, the latter being the direct sum of the horizontal spaces of the vector bundle ${\mathcal A}$ at $J_1$ and $J_2$. Note also that the horizontal space of ${\mathcal A}$ at $J_k$, $k=1,2$, is the horizontal space of ${\mathcal Z}$ at $J_k$. Then if
$X\in T_{\pi(J)}M$, denoting the horizontal lift of $X$ to $T_{J_k}{\mathcal Z}$ by $X^{\widetilde h}_{J_k}$, we have $X^h_J=(X^{\widetilde h}_{J_1},X^{\widetilde h}_{J_2})$.  Let  $X,Y,Z\in T_{\pi(J)}M$, $U_1,V_1,W_1\in{\mathcal V}_{J_1}$ (the vertical subspace of $T_{J_1}{\mathcal Z}$). Then for $U=(U_1,0), V=(V_1,0), W=(W_1,0)\in {\mathcal V}_J$, for every ${\bf t}=(t_1,t_2)$, $t_1,
t_2>0$, and for $n=1,2$,
$$
\begin{array}{c}
(\widetilde{D}_{Z^{\widetilde h}_{J_1}+W_1}\widetilde{\Omega}_{t_1,1})(X^{\widetilde h}_{J_1}+U_1,Y^{\widetilde h}_{J_1}+V_1)= (D_{Z^h_J+W}\Omega_{{\bf
t},n})(X^h_J+U,Y^h_J+V), \\[6pt]

d\widetilde{\Omega}_{t_1,1}(X^{\widetilde h}_{J_1}+U_1,Y^{\widetilde h}_{J_1}+V_1,Z^{\widetilde h}_{J_1}+W_1)=d\Omega_{{\bf t},n}(X^h_J+U,Y^h_J+V,Z^h_J+W),\\[6pt]

\delta\widetilde {\Omega}_{t_1,1}(X^{\widetilde h}_{J_1}+U_1)=\delta \Omega_{{\bf t},n}(X^h_J+U),
\end{array}
$$
Similar identities holds for the covariant derivatives, differentials and codifferentials of $\widetilde{\Omega}_{t_1,2}$ and
$\Omega_{{\bf t},n}$ for $n=3,4$. Note also that
$$
h_{t_1}(X^{\widetilde h}_{J_1}+U_1,Y^{\widetilde h}_{J_1}+V_1)=H_{(t_1,t_2)}(X^h_J+U,Y^h_J+V).
$$
These relations and the defining identities for the Gray-Hervella classes \cite[p. 41]{GH} imply the following.

\begin{prop}\label{GH-rel} If one of the almost Hermitian structures $(H_{\bf t},{\mathscr J}^1)$ and
$(H_{\bf t},{\mathscr J}^2)$ on ${\mathcal G}$ belongs to a (non-trivial) Gray-Hervella class, then the almost Hermitian
structure $(h_{t_1},{\mathcal J}_1)$ on ${\mathcal Z}$ belongs to the same class.
\newline
If one of $(H_{\bf t},{\mathscr J}^3)$ and $(H_{\bf t},{\mathscr J}^4)$ is in a Gray-Hervella class, $(h_{t_1},{\mathcal
J}_2)$ is in the same class.
\end{prop}

 Suppose that the manifold $M$ is oriented and denote by ${\mathcal Z}_{\pm}$ the bundle over $M$ whose sections
are the almost complex structures compatible with the metric and $\pm$ the orientation of $M$. The bundles ${\mathcal Z}_{+}$ and ${\mathcal Z}_{-}$ are usually called the positive and the negative twistor space of
$(M,g)$, respectively. The spaces ${\mathcal
Z}_{\pm}$ are disjoint open subsets  of the twistor space ${\mathcal Z}$.
If $M$ is connected and oriented, they are the connected components of the manifold ${\mathcal Z}$.
Correspondingly, ${\mathcal G}={\mathcal
Z}\times_{M} {\mathcal Z}$ has four connected components, the product bundles ${\mathcal Z}_{+}\times_{M}{\mathcal
Z}_{+}$, ${\mathcal Z}_{+}\times_{M}{\mathcal Z}_{-}$, ${\mathcal Z}_{-}\times_{M}{\mathcal Z}_{+}$, ${\mathcal
Z}_{-}\times_{M}{\mathcal Z}_{-}$.

Clearly, Proposition~\ref{GH-rel} holds for the restrictions of $(H_{\bf t},{\mathscr J}^n)$ to ${\mathcal Z}_{+}\times_{M}
{\mathcal Z}_{+}$ and ${\mathcal Z}_{+}\times_{M} {\mathcal Z}_{-}$ (resp. ${\mathcal Z}_{-}\times_{M} {\mathcal Z}_{+}$ and
${\mathcal Z}_{-}\times_{M} {\mathcal Z}_{-}$) and the restriction of $(h_{t_1}{\mathcal J}_k)$ to ${\mathcal Z}_{+}$ (resp.
${\mathcal Z}_{-})$.
\smallskip

Recall that the curvature operator has the decomposition \cite[G (1.116)]{Besse}.
$$
{\mathcal R}=\frac{s}{n(n-1)}Id +{\mathcal B} + {\mathcal W}
$$
where $s=s_M$ is the scalar curvature, the operator ${\mathcal B}$ corresponds to the traceless Ricci tensor, and
${\mathcal W}$ corresponds to the Weyl conformal tensor. As is well-known, if $M$ is oriented and four-dimensional, the
operator ${\mathcal W}$ has an extra decomposition which can be described by means of the Hodge star operator.
In dimension four, the Hodge operator acts on $\Lambda^2TM$ as an involution and we have the orthogonal decomposition
$$
\Lambda^2TM=\Lambda^2_{+}TM\oplus \Lambda^2_{-}TM
$$
where $\Lambda^2_{\pm}TM$ are the subbundles corresponding to the $\pm 1$-eigenvalues of the Hodge operator. Denoting the
restriction of ${\mathcal W}$ to $\Lambda^2_{\pm}TM$ by ${\mathcal W}_{\pm}$, we obtain the decomposition (\cite{ST}, \cite[G sec. 1.128]{Besse})
\begin{equation}\label{cur-decom}
{\mathcal R}=\frac{s}{12}Id +{\mathcal B} + {\mathcal W}_{+}+{\mathcal W}_{-}.
\end{equation}
Note that ${\mathcal B}$ sends $\Lambda^2_{\pm}TM$ to $\Lambda^2_{\mp}TM$, ${\mathcal W}_{\pm}$ sends $\Lambda^2_{\pm}TM$ to $\Lambda^2_{\pm}TM$, and ${\mathcal W}_{\pm}|\Lambda^2_{\mp}TM=0$.
Note also that changing the orientation of $M$ interchanges the roles of $\Lambda^2_{+}TM$ and $\Lambda^2_{-}TM$,
respectively the roles of ${\mathcal W}_{+}$ and ${\mathcal W}_{-}$.

The manifold $(M,g)$ is Einstein exactly when ${\mathcal B}=0$. It is called anti-self-dual (self-dual) when ${\mathcal
W}_{+}=0$ (${\mathcal W}_{-}=0$).

Recall that by the Atiyah-Hitchin-Singer theorem \cite{AHS} the almost complex structure ${\mathcal J}_1$ restricted to
${\mathcal Z}_+{}$ (resp. ${\mathcal Z}_{-}$) is integrable if and only if the base manifold $(M,g)$ is anti-self-dual
(resp. self-dual). In contrast, the almost complex structure ${\mathcal J}_2$ is never integrable by a result
of Eells-Salamon \cite{ES}.

The isometry ${\mathcal A}\cong\Lambda^{2}TM$, $a\to a^{\wedge}$, identifies the bundle ${\mathcal Z}_{\pm}$ with the
sphere subbundle of $\Lambda^{2}_{\pm}TM$ of radius $\sqrt 2$. Under this isomorphism, if $J\in{\mathcal Z}_{\pm}$, the
vertical space ${\mathcal V}_J$ at $J$ is identified  with the orthogonal complement $({\mathbb R}J^{\wedge})^{\perp}$ of
${\mathbb R}J^{\wedge}$ in $\Lambda^{2}_{\pm}TM$.

\medskip

\noindent {\it Convention}. Henceforth, we assume that $M$ is oriented and of dimension four. We shall freely identify
skew-symmetric endomorphisms  with $2$-forms and shall use the same notation for both of them.
$\hfill\Box$

\smallskip

\noindent {\it Notation}. It is convenient to set  ${\mathcal G}_{++}={\mathcal Z}_{+}\times_{M} {\mathcal Z}_{+}$, ${\mathcal G}_{+-}={\mathcal Z}_{+}\times_{M} {\mathcal Z}_{-}$, and similar for ${\mathcal G}_{--}$, ${\mathcal G}_{-+}$.
$\hfill\Box$

Changing the orientation of $M$ interchanges the roles of ${\mathcal G}_{++}$ and ${\mathcal G}_{--}$, as well as  ${\mathcal G}_{+-}$ and
${\mathcal G}_{-+}$. Hence it is enough to consider only ${\mathcal G}_{++}$ and ${\mathcal G}_{+-}$.

\smallskip

According to \cite[Theorem 4.6]{M89} the possible Gray-Hervella classes for $(h_{t_1},{\mathcal J}_1)$ considered on
${\mathcal Z}_{+}$ are ${\mathscr K}$  and ${\mathscr W}_3$. Thus Proposition~\ref{GH-rel} implies the following

\begin{cor}
The only  possible Gray-Hervella classes for $(H_{\bf t},{\mathscr J}^1)$ and $(H_{\bf t},{\mathscr J}^2)$ restricted to
${\mathcal G}_{++}={\mathcal Z}_{+}\times_{M} {\mathcal Z}_{+}$ or ${\mathcal G}_{+-}={\mathcal Z}_{+}\times_{M} {\mathcal Z}_{-}$ are ${\mathscr K}$ and ${\mathscr
W}_3$.
\end{cor}

\noindent {\bf Remark}. The class ${\mathscr K}$ consists of K\"ahler structures; ${\mathscr W}_3$ is the class of
Hermitian semi-K\"ahler structures (for the defining condition of this class see the proof of the Theorem~\ref{GH-1,2-W3} below).

\newpage

\begin{theorem}\label{GH-1,2-K}\noindent
\begin{itemize}  \item [(a)] $(H_{\bf t},{\mathscr J}^n)$, $n=1$ or $2$, restricted to
${\mathcal G}_{++}$  does not belong to the class ${\mathscr K}$.

\item[(b)] $(H_{\bf t},{\mathscr J}^n)$, $n=1$ or $2$, restricted to ${\mathcal G}_{+-}$  belongs to
the class ${\mathscr K}$ if and only if $(M,g)$ is Einstein with positive scalar curvature $s$, anti-self-dual,
${\mathcal W}_{-}(\tau)=-\displaystyle{\frac{s}{12}}\tau$ for $\tau\in\Lambda^2_{-}TM$, and $t_1=\displaystyle{\frac{6}{s}}$.

\end{itemize}
\end{theorem}

\begin{proof} Suppose that one of the  structures $(H_{\bf t},{\mathscr J}^n)$, $n=1$ or $2$,  restricted to \\
${\mathcal Z}_{+}\times_{M} {\mathcal Z}_{+}$ or ${\mathcal Z}_{+}\times_{M} {\mathcal Z}_{-}$ belongs to the class ${\mathscr
K}$.  Then so does the restriction of $(h_{t_1},{\mathcal J}_1)$ to ${\mathcal Z}_{+}$. It is well-known that
$(h_{t_1},{\mathcal J}_1)$ is K\"ahler on ${\mathcal Z}_{+}$ if and only if the the base manifold $(M,g)$ is Einstein
with positive scalar curvature, anti-self-dual, and $\displaystyle{t_1=\frac{6}{s}}$ \cite{FK} (see also \cite{M89}, but
be aware that the metric on $\Lambda^2TM$ used there is one half of the metric here and in \cite{FK}, hence the curvature operato ${\mathcal R}$ is one half of the curvature operator in \cite{M89}). Then by (\ref{vhh})
\begin{equation}\label{K1}
g({\mathcal R}(J_2V_2),X\wedge J_1Y+J_1X\wedge Y)=0
\end{equation}
for every $J=(J_1,J_2)\in{\mathcal Z}_{+}\times_{M} {\mathcal Z}_{+}$ or ${\mathcal Z}_{+}\times_{M} {\mathcal Z}_{-}$, $X,Y\in
T_{\pi(J)}M$, $V_2\in{\mathcal V}_{J_2}$, where ${\mathcal R}(J_2V_2)=\displaystyle{\frac{s}{12}}J_2V_2+{\mathcal W}_{-}(J_2V_2)$.  Note
that the $2$-vector $X\wedge J_1Y+J_1X\wedge Y\in\Lambda^2_{+}T_{\pi(J)}M$. This follows from the easily verifiable fact
that $X\wedge J_1Y+J_1X\wedge Y$ is orthogonal to $\Lambda^2_{-}T_pM$, $p=\pi(J)$, since $J_1$ commutes with every
endomorphism of $T_pM$ lying in $\Lambda^2_{-}T_pM$.

\smallskip

$(a)$. If $J_2\in\Lambda^2_{+}TM$, identity (\ref{K1}) takes the form
$$
g(J_2V_2,X\wedge J_1Y+J_1X\wedge Y)=0.
$$
The latter identity is not satisfied, for example, for $J={\sqrt 2}(s_1^{+},s_2^{+})$, $V_2=s_1^{+}$, $(X,Y)=(E_1,E_3)$
where $E_1,...,E_4$ is an oriented orthonormal basis of $T_pM$ and $\{s_i^{+}\}$ are defined via (\ref{s-basis}).
Therefore none of the restriction  of  $(H_{\bf t},{\mathscr J}^1)$ and
$(H_{\bf t},{\mathscr J}^2)$ to ${\mathcal G}_{++}$ is a K\"ahler structure.

\smallskip

 $(b)$. If $J_2\in\Lambda^2_{-}TM$, identity (\ref{K1}) is automatically satisfied since ${\mathcal
R}(JV_2)\in\Lambda^2_{-}T_pM$. We also have by (\ref{hhv})
$$
g({\mathcal R}(V_2),Z\wedge X)=g({\mathcal R}(J_2V_2),Z\wedge J_1X).
$$
This implies
$$
g({\mathcal R}(V_2),Z\wedge X+J_1Z\wedge J_1X)=g({\mathcal R}(J_2V_2),Z\wedge J_1X-J_1Z\wedge X),
$$
$$
g({\mathcal R}(V_2),Z\wedge X- J_2Z\wedge J_2X)=g({\mathcal R}(J_2V_2),Z\wedge J_1X-J_2Z\wedge J_1J_2 X).
$$
Set $J_1={\sqrt 2}s_1^{+}$, $(Z,X)=(E_1,E_3)$ in the first of the latter identities and also  $(J_1,J_2)={\sqrt
2}(-s_1^{+},s_1^{-})$, $(Z,X)=(E_1,E_3)$. Thus, we obtain
$$
g({\mathcal R}(V_2),s_2^{-})=g({\mathcal R}(J_2V_2),s_3^{-}),\quad g({\mathcal R}(V_2),s_2^{-})=-g({\mathcal
R}(J_2V_2),s_3^{-}),
$$
where $J_2={\sqrt 2}s_1^{-}$ and $V_2\perp s_1^{-}$. Hence
$$
g({\mathcal R}(s_2^{-}),s_2^{-})=g({\mathcal R}(s_3^{-}),s_2^{-})=0.
$$
Replacing the basis $E_1,E_2,E_3,E_4$ with the bases $E_1,E_3,E_4,E_2$ and $E_1,E_4,E_2,E_3$, we see that
$$
g({\mathcal R}(s_3^{-}),s_3^{-})=g({\mathcal R}(s_1^{-}),s_3^{-})=0, \quad g({\mathcal R}(s_1^{-}),s_1^{-})=g({\mathcal
R}(s_2^{-}),s_1^{-})=0.
$$
It follows that ${\mathcal R}|\Lambda^2_{-}TM=0$. Thus, if one of the restrictions of the structures $(H_{\bf t},{\mathscr
J}^1)$ and $(H_{\bf t},{\mathscr J}^2)$ to ${\mathcal Z}_{+}\times_{M} {\mathcal Z}_{-}$  is K\"ahler, then $(M,g)$ is
Einstein with positive scalar curvature $s$, anti-self-dual, ${\mathcal
W}_{-}=-\displaystyle{\frac{s}{12}}Id|\Lambda^2_{-}TM$, and $\displaystyle{t_1=\frac{6}{s}}$.

Conversely, if $(M,g)$ satisfies these conditions, it follows easily from Proposition~\ref{DF} that the restrictions of
$(H_{\bf t},{\mathscr J}^1)$ and $(H_{\bf t},{\mathscr J}^2)$  to  ${\mathcal G}_{+-}$ are
K\"ahler structures.
\end{proof}

\begin{theorem}\label{GH-1,2-W3}\noindent
\begin{itemize}
\item [(a)] $(H_{\bf t},{\mathscr J}^n)$, $n=1$ or $2$, restricted to ${\mathcal G}_{++}$
 is in the class ${\mathscr W}_3$ if and only if $(M,g)$ is anti-self-dual and
scalar flat.

\item[(b)] $(H_{\bf t},{\mathscr J}^n)$, $n=1$ or $2$, restricted to ${\mathcal G}_{+-}$  is in
the class ${\mathscr W}_3$ if and only if $(M,g)$ is anti-self-dual.

\end{itemize}

\end{theorem}

\begin{proof} $(a)$. Suppose that one of the restrictions of $(H_{\bf t},{\mathscr J}^1)$ and $(H_{\bf t},{\mathscr J}^2)$ to \\
${\mathcal Z}_{+}\times_{M}{\mathcal Z}_{+}$ belongs to the class ${\mathscr W}_3$. Recall that the defining conditions for this class
are: $(i)$ the Nijenhuis tensor of the almost complex structure vanishes; $(ii)$ the codifferential of the fundamental
$2$-form vanishes. By Proposition~\ref{GH-rel},  the almost Hermitian structure $(h_{t_1},{\mathcal J}_1)$ on ${\mathcal Z}_{+}$ is of class ${\mathscr W}_3$. In particular, the Atiyah-Hitchin-Singer structure
${\mathcal J}_1$ on ${\mathcal Z}_{+}$ is integrable. Therefore $(M,g)$ is anti-self-dual. Then, by Corollary~\ref{delta
F}, condition $(ii)$ takes the form $g({\mathcal R}(J_2V_2),J_1)=0$, Thus
$$
\frac{s}{12}g(J_2V_2,J_1)=0.
$$
for every $J_1,J_2\in{\mathcal Z}_{+}$ and every  $V_2\in{\mathcal V}_{J_2}$.   This implies $s=0$.

Conversely, if $(M,g)$ is anti-self-dual and scalar flat, then conditions $(i)$ and $(ii)$ are fulfilled by
Corollaries~\ref{Nij} and \ref{delta F}.

\smallskip

$(b)$. Suppose that one of the restrictions of $(H_{\bf t},{\mathscr J}^1)$ and $(H_{\bf t},{\mathscr J}^2)$ to
${\mathcal Z}_{+}\times_{M} {\mathcal Z}_{-}$ belongs to the class ${\mathscr W}_3$. Then $(M,g)$ is anti-self-dual.
Conversely, if $(M,g)$ is anti-self dual, ${\mathcal R}|\Lambda^2_{-}TM$ takes its values in $\Lambda^2_{-}TM$. Then
conditions $(i)$ and $(ii)$ are satisfied by Corollaries~\ref{Nij} and \ref{delta F} since $X\wedge J_1Y+J_1X\wedge Y$
and $X\wedge Y-J_1X\wedge J_1Y$ lies in $\Lambda^2_{+}TM$.

\end{proof}

Proposition~\ref{GH-rel} and \cite[Theorem 4.7]{M89}  imply the following

\begin{cor}
The only  possible Gray-Hervella classes for $(H_{\bf t},{\mathscr J}^3)$ and $(H_{\bf t},{\mathscr J}^4)$ restricted to
${\mathcal G}_{++}={\mathcal Z}_{+}\times {\mathcal Z}_{+}$ or ${\mathcal G}_{+-}={\mathcal Z}_{+}\times {\mathcal Z}_{-}$ are ${\mathscr W}_1\oplus
{\mathscr W}_2\oplus {\mathscr W}_3$, ${\mathscr W}_1\oplus {\mathscr W}_2$, ${\mathscr W}_1\oplus {\mathscr W}_3$,
${\mathscr W}_2\oplus {\mathscr W}_3$, ${\mathscr W}_1$, ${\mathscr W}_2$.
\end{cor}

\noindent {\bf Remark}. ${\mathscr W}_1\oplus {\mathscr W}_2\oplus {\mathscr W}_3$ is the class of semi-K\"ahler almost
Hermitian structures, ${\mathscr W}_1\oplus {\mathscr W}_2$ of the quasi-K\"ahler ones, ${\mathscr W}_1$ is the class  of
nearly- K\"ahler structures, ${\mathscr W}_2$ of almost-K\"ahler ones. The defining conditions of these classes are given in the proofs of the next claims. .

\begin{theorem}\label{GH-3,4-W1+W2+W3}\noindent
\begin{itemize}
\item [(a)] $(H_{\bf t},{\mathscr J}^n)$, $n=3$ or $4$, restricted to
${\mathcal G}_{++}$ is in the class ${\mathscr W}_1\oplus {\mathscr W}_2\oplus {\mathscr W}_3$, if
and only if $(M,g)$ is anti-self-dual and scalar flat.

\item[(b)] $(H_{\bf t},{\mathscr J}^n)$, $n=3$ or $4$, restricted to ${\mathcal G}_{+-}$  is in
the class ${\mathscr W}_1\oplus {\mathscr W}_2\oplus {\mathscr W}_3$ if and only if $(M,g)$ is anti-self-dual.

\end{itemize}

\end{theorem}

\begin{proof}
The class ${\mathscr W}_1\oplus {\mathscr W}_2\oplus {\mathscr W}_3$ is defined by the requirement that the
codifferential of the fundamental $2$-form vanishes. For $n=3$ or $4$, if one of the restrictions of $(H_{\bf t},{\mathscr J}^n)$ to ${\mathcal Z}_{+}\times_{M} {\mathcal Z}_{+}$ belongs to this class, then $(M,g)$ is anti-self-dual by
Proposition~\ref{GH-rel} and \cite[Theorem 4.7 (i)]{M89}. Then, as we have seen in the proof of
Theorem~\ref{GH-1,2-W3}: $(a)$ $(M,g)$ is scalar flat when we consider the defining condition on ${\mathcal
Z}_{+}\times_{M}{\mathcal Z}_{+}$, $(b)$ no other restrictions in the case of ${\mathcal Z}_{+}\times_{M}{\mathcal Z}_{-}$. This
proofs the "if" part of the theorem. The converse statement follows immediately from Corollary~\ref{delta F}.

\end{proof}

\begin{theorem}\label{GH-3,4-W1+W2}
\noindent
\begin{itemize}
\item [(a)] $(H_{\bf t},{\mathscr J}^n)$, $n=3$ or $4$, restricted to
${\mathcal G}_{++}$ is in the class ${\mathscr W}_1\oplus {\mathscr W}_2$ if and only if $(M,g)$
is anti-self-dual and Ricci flat

\item[(b)] $(H_{\bf t},{\mathscr J}^n)$, $n=3$ or $4$, restricted to ${\mathcal G}_{+-}$  is in
the class ${\mathscr W}_1\oplus {\mathscr W}_2$ if and only if $(M,g)$  is Einstein, anti-self-dual, and ${\mathcal
W}_{-}(\tau)=-\displaystyle{\frac{s}{12}}\tau$ for $\tau\in\Lambda^2_{-}TM$.

\end{itemize}

\end{theorem}

\begin{proof}

Suppose that one of the restrictions of $(H_{\bf t},{\mathscr J}^3)$  and $(H_{\bf t},{\mathscr J}^4)$ to ${\mathcal
Z}_{+}\times_{M} {\mathcal Z}_{+}$ or ${\mathcal Z}_{+}\times_{M} {\mathcal Z}_{-}$ belongs to ${\mathscr W}_1\oplus {\mathscr
W}_2$. The defining condition of this class reads as
$$
(D_{A}\Omega_{{\bf t},n})(B,C)+(D_{{\mathscr J}^nA}\Omega_{{\bf t},n})({\mathscr J}^nB,C)=0,\quad A,B,C\in T{\mathcal G},\quad
n=3,4.
$$
By Proposition~\ref{GH-rel} and \cite[Theorem 4.7 (ii)]{M89}, $(M,g)$ is Einstein and anti-self-dual. Under these
conditions, it is easy to see by means of Proposition~\ref{DF} that the defining condition is equivalent to
\begin{equation}\label{1+2}
\begin{array}{c}
g({\mathcal R}(V_2),X\wedge Y)+g({\mathcal R}(J_2V_2),X\wedge J_1Y)=0,\quad \rm{for}~n=3,\\[6pt]
g({\mathcal R}(V_2),X\wedge Y)-g({\mathcal R}(J_2V_2),J_1X\wedge Y)=0,\quad \rm{for}~n=4.
\end{array}
\end{equation}

\smallskip

$(a)$. If $V_2\in\Lambda^2_{+}TM$, the latter identities become
$$
\frac{s}{12}[g(V_2X,Y)+g((J_2V_2)X,J_1Y)=0, \quad \frac{s}{12}[g(V_2X,Y)+g((J_2V_2)Y,J_1X)=0.
$$
Setting $J_1={\sqrt 2}s_1^{+}$, $J_2={\sqrt 2}s_2^{+}$, $V_2=s_1^{+}$, $(X,Y)=(E_3,E_3)$ for an oriented orthonormal
basis $E_1,...,E_4$, we see that $s=0$.

\smallskip

$(b)$. If $V_2\in\Lambda^2_{-}TM$, identities (\ref{1+2}) with $J_1={\sqrt 2}s_1^{+}$ and $(X,Y)=(E_1,E_2)$ imply
$g({\mathcal R}(V_2),E_1\wedge E_2)=0$. It follows that ${\mathcal R}(V_2)=0$ for every $V_2\in\Lambda^2_{-}TM$.

This proves the necessity of the curvature conditions stated in the theorem. Their sufficiency is an easy consequence of
Proposition~\ref{DF}.

\end{proof}

\begin{theorem}\label{GH-3,4-W1+W3}
\noindent
\begin{itemize}
\item [(a)] $(H_{\bf t},{\mathscr J}^n)$, $n=3$ or $4$, restricted to
${\mathcal G}_{++}$ does not belong to the class ${\mathscr W}_1\oplus {\mathscr W}_3$.

\item[(b)] $(H_{\bf t},{\mathscr J}^n)$, $n=3$ or $4$, restricted to ${\mathcal G}_{+-}$  belongs
to the class ${\mathscr W}_1\oplus {\mathscr W}_3$ if and only if $(M,g)$ is Einstein with positive scalar curvature $s$,
anti-self-dual, and $t_1=\displaystyle{\frac{3}{s}}$.

\end{itemize}

\end{theorem}

\begin{proof}
The defining conditions of the class ${\mathscr W}_{1}\oplus {\mathscr W}_{3}$ in the case of $(H_{\bf t},{\mathscr
J}^n)$ read as
\begin{equation}\label{eq1-1+3}
(D_{A}\Omega_{{\bf t},n})(A,C)+(D_{{\mathscr J}^nA}\Omega_{{\bf t},n})({\mathscr J}^nA,C)=0,\quad \delta \Omega_{{\bf t},n}=0,\quad A,C\in
T{\mathcal G},\quad n=3,4.
\end{equation}
If one of the restrictions of $(H_{\bf t},{\mathscr J}^n)$ to ${\mathcal Z}_{+}\times_{M} {\mathcal Z}_{+}$ or ${\mathcal
Z}_{+}\times_{M} {\mathcal Z}_{-}$ belongs to  ${\mathscr W}_{1}\oplus {\mathscr W}_{3}$, then $(M,g)$ is anti-self dual with
positive scalar curvature $s=\displaystyle{\frac{3}{t_1}}$ by Proposition~\ref{GH-rel} and \cite[Theorem 4.7 (iii)]{M89}. Then
identities (\ref{eq1-1+3}) are equivalent to
\begin{equation}\label{eq2-1+3}
\begin{array}{c}
(-1)^ng({\mathcal R}(V_2),X\wedge Y-J_1X\wedge J_1Y)+g({\mathcal R}(J_2V_2),X\wedge J_1Y+J_1X\wedge Y)=0,\\[6pt]
g({\mathcal R}(J_2V_2),J_1)=0
\end{array}
\end{equation}
for every $J=(J_1,J_2)\in {\mathcal G}_{++}$ or ${\mathcal G}_{+-}$, $X,Y\in T_{\pi(J)}M$, $V_2\in{\mathcal V}_{J_2}$.

For $J\in {\mathcal G}_{++}$ and $V_2\in\Lambda^2_{+}TM$, the second equation of (\ref{eq2-1+3}) becomes
$$
\frac{s}{12}g(J_2V_2,J_1)=0.
$$
This implies $s=0$, a contradiction. This proves $(a)$.

If $J\in {\mathcal G}_{+-}$ and $V_2\in\Lambda^2_{-}TM$, then the second equation of (\ref{eq2-1+3}) becomes
$$
\frac{s}{12}g({\mathcal B}(J_2V_2),J_1)=0.
$$
This is equivalent to ${\mathcal B}=0$. The  $2$-vectors $X\wedge Y-J_1X\wedge J_1Y$  and  $X\wedge J_1Y+J_1X\wedge Y$ in
the first identity of (\ref{eq2-1+3}) lie in $\Lambda^2_{+}TM$. Hence this identity is always satisfied when ${\mathcal
B}=0$ since in this case ${\mathcal R}$ preserves $\Lambda^2_{-}TM$. This proves $(b)$.

\end{proof}

\begin{theorem}\label{GH-3,4-W2+W3}\noindent
\begin{itemize}
\item [(a)] $(H_{\bf t},{\mathscr J}^n)$, $n=3$ or $4$, restricted to
${\mathcal G}_{++}$ does not belong to the class ${\mathscr W}_2\oplus {\mathscr W}_3$.

\item[(b)] $(H_{\bf t},{\mathscr J}^n)$, $n=3$ or $4$, restricted to ${\mathcal G}_{+-}$  belongs
to the class ${\mathscr W}_2\oplus {\mathscr W}_3$ if and only if $(M,g)$ is Einstein with negative scalar curvature $s$,
anti-self-dual, and $t_1=\displaystyle{-\frac{6}{s}}$.

\end{itemize}

\end{theorem}

\begin{proof}
The belonging conditions for the class ${\mathscr W}_{2}\oplus {\mathscr W}_{3}$ are
\begin{equation}\label{eq1 2+3}
\displaystyle \mathop{\mathfrak{S}}_{A,B,C}\{(D_{A}\Omega_{{\bf t},n})(A,C)-(D_{{\mathscr J}^nA}\Omega_{{\bf t},n})({\mathscr
J}^nA,C)\}=0,\quad \delta \Omega_{{\bf t},n}=0,\quad A,C\in T{\mathcal G},\quad n=3,4.
\end{equation}
If one of the restrictions of $(H_{\bf t},{\mathscr J}^n)$ to ${\mathcal Z}_{+}\times_{M}{\mathcal Z}_{+}$ or ${\mathcal
Z}_{+}\times_{M}{\mathcal Z}_{-}$ belongs to  ${\mathscr W}_{1}\oplus {\mathscr W}_{3}$, then $(M,g)$ is anti-self dual with
negative scalar curvature $s=-\displaystyle{\frac{6}{t_1}}$ by Proposition~\ref{GH-rel} and \cite[Theorem 4.7 (iv)]{M89}. Then
conditions (\ref{eq1 2+3}) reduce to
\begin{equation}\label{eq2 2+3}
\begin{array}{c}
\displaystyle \mathop{\mathfrak{S}}\{(-1)^ng({\mathcal R}(U_2),Y\wedge Z-J_1Y\wedge J_1Z)+ g({\mathcal
R}(J_2U_2),J_1Y\wedge Z+Y\wedge J_1Z)\}=0\\[6pt]
g({\mathcal R}(J_2V_2),J_1)=0,
\end{array}
\end{equation}
where the cyclic sum is over $(U_2,Y,Z)$, $(V_2,Z,X)$, $(W_2,X,Y)$ for $J=(J_1,J_2)$, $X,Y,Z\in T_{\pi(J)}M$,
$U_2,V_2,W_2\in{\mathcal V}_{J_2}$.

For  the restriction of $(H_{\bf t},{\mathscr J}^n)$ to ${\mathcal Z}_{+}\times_{M}{\mathcal Z}_{+}$, $n=1,2$, the second
identity of (\ref{eq2 2+3}) cannot be always satisfied since $s\neq 0$. This remark proves $(a)$. The second identity of (\ref{eq2 2+3})  implies ${\mathcal B}=0$ when
considering the restriction of $(H_{\bf t},{\mathscr J}^n)$ to ${\mathcal Z}_{+}\times_{M} {\mathcal Z}_{-}$, $n=3,4$. In this case, the
first identity of (\ref{eq2 2+3}) is satisfied when ${\mathcal B}=0$. This proves $(b)$.

\end{proof}

\begin{theorem}\label{GH-3,4-W1}\noindent
\begin{itemize}
\item [(a)] $(H_{\bf t},{\mathscr J}^n)$, $n=3$ or $4$, restricted to
${\mathcal G}_{++}$ does not belong to the class ${\mathscr W}_1$.

\item[(b)] $(H_{\bf t},{\mathscr J}^n)$, $n=3$ or $4$, restricted to ${\mathcal G}_{+-}$  belongs
to the class ${\mathscr W}_1$ if and only if $(M,g)$ is Einstein with positive scalar curvature $s$, anti-self-dual,
${\mathcal W}_{-}(\tau)=-\displaystyle{\frac{s}{12}}\tau$ for $\tau\in\Lambda^2_{-}TM$, and
$t_1=\displaystyle{\frac{3}{s}}$.

\end{itemize}

\end{theorem}

\begin{proof}
The claim $(a)$ directly  follows from Theorem~\ref{GH-3,4-W1+W3} $(a)$.

By Theorems~\ref{GH-3,4-W1+W2} $(b)$ and \ref{GH-3,4-W1+W3} $(b)$, if $(H_{\bf t},{\mathscr J}^n)$, $n=3$ or $4$,
restricted to ${\mathcal Z}_{+}\times_{M}{\mathcal Z}_{-}$  belongs to ${\mathscr W}_1$, then $(M,g)$ is Einstein,
anti-self-dual,  ${\mathcal R}|\Lambda^2_{-}TM=0$, and $t_1=\displaystyle{\frac{3}{s}}$. It is easy to check by means
of Proposition~\ref{DF} that under these conditions
$$
(D_{A}\Omega_{{\bf t},n})(A,B)=0.
$$
This proves $(b)$ since the latter identity is the defining condition for the class ${\mathscr W}_1$.

\end{proof}

Theorems~\ref{GH-3,4-W1+W2}  and \ref{GH-3,4-W2+W3} imply the following statement whose proof is similar to that of the
preceding theorem.

\begin{theorem}\label{GH-3,4-W2}\noindent
\begin{itemize}
\item [(a)] $(H_{\bf t},{\mathscr J}^n)$, $n=3$ or $4$, restricted to
${\mathcal G}_{++}$ does not belong to the class ${\mathscr W}_2$.

\item[(b)] $(H_{\bf t},{\mathscr J}^n)$, $n=3$ or $4$, restricted to ${\mathcal G}_{+-}$  belongs
to the class ${\mathscr W}_2$ if and only if $(M,g)$ is Einstein with negative scalar curvature $s$, anti-self-dual,
${\mathcal W}_{-}(\tau)=-\displaystyle{\frac{s}{12}}\tau$ for $\tau\in\Lambda^2_{-}TM$, and
$t_1=\displaystyle{-\frac{6}{s}}$.

\end{itemize}

\end{theorem}

\end{document}